\documentclass{amsart}

 \usepackage{amsxtra}
 \usepackage{amsthm}
 \usepackage{amsmath}
 \usepackage{amsthm}
 \usepackage{amsfonts}
 \usepackage{amssymb}
 \usepackage[all]{xy}
\usepackage[notref,notcite]{showkeys} 

\input xy
\xyoption{all}

\newtheoremstyle{mytheorem}{.5em}{.5em}%
     {\it}
     {}
     {}
     {}
     {.5em}
     {#2 \thmname{ \bf{#1}.} \thmnote{\it{#3}.}}

\theoremstyle{mytheorem}
\newtheorem{theorem}{Theorem}[section]
\newtheorem{lemma}[theorem]{Lemma}
\newtheorem{corollary}[theorem]{Corollary}

\newtheoremstyle{note}{.5em}{.5em}%
     {}
     {}
     {}
     {}
     {.5em}
     {#2 \thmname{ \bf{#1}.} \thmnote{\it{#3}.}}

\theoremstyle{note}
\newtheorem{ntt}[theorem]{}
\newtheorem{example}[theorem]{Example}
\newtheorem{remark}[theorem]{Remark}
\newtheorem{definition}[theorem]{Definition}

\newtheorem{assumption}[theorem]{Assumption}

\newcommand{\calJ}{\mathcal J}

\newcommand{\CH}{\operatorname{CH}}

\newcommand{\La}{\Lambda}           
\newcommand{\la}{\lambda}            
\newcommand{\om}{\omega}          
\newcommand{\Spec }{\operatorname{Spec}}

\newcommand{\IF}{\mathcal{I}_F}    

\newcommand{\ttt}{\mathtt t}     
\newcommand{\hh}{\mathtt{h}}       

\newcommand{\ep}{\epsilon}
\newcommand{\Th}{\Theta}

\newcommand{\al}{\alpha}

\newcommand{\bbR}{\mathbb R}
\newcommand{\calI}{\mathcal I}
\newcommand{\calL}{\mathcal L}

\newcommand{\Z}{\mathbb Z}

\title{on the $\gamma$-Filtration of Oriented Cohomology of Complete Spin-flags}

\date{\today}

\author{Changlong Zhong}

\address{Changlong Zhong, Department of Mathematics and Statistics,
University of Ottawa}
\email{zhongusc@gmail.com}

\begin{document}
\maketitle

\begin{abstract}We study the characteristic map of algebraic oriented cohomology of complete spin-flags and the ideal of invariants of formal group algebra. As an application, we provide an annihilator of the torsion part of the $\gamma$-filtration. Moreover,  if the formal group law determined by the oriented cohomology theory is congruent to the additive formal group law modulo 2, then at degree 2 and 3, the $\gamma$-filtration of complete spin-flags is torsion free. 
\end{abstract}

\section{Introduction}Oriented cohomology theories \cite{LM} of algebraic varieties over base field $k$ are cohomology theories generalized from the Chow group $\CH$ and the Grothendieck group $K_0$. They are algebraic analogue of  cohomology theories of complex manifolds. In particular, each  oriented cohomology theory $\hh$  determines a one-dimensional formal group law $F$ over the coefficient ring $R=\hh(\Spec k)$. For example, $\CH$ (resp. $K_0$) determines the additive formal group law $F_a$ (resp. the multiplicative formal group law $F_m$).

Given a  split simple simply connected linear algebraic group $G$ with  the variety of complete flags $X$ and a fixed maximal torus $T$, let $W$ be its Weyl group and $\La$ be the weight lattice with respect to $T$. For arbitrary oriented cohomology $\hh$ and corresponding formal group law $F$,  Calm\`es-Petrov-Zainoulline \cite{CPZ} construct a formal group algebra $R[[\La]]_F$ and a characteristic map $c_F: R[[\La]]_F\to \hh(X)$. These constructions generalize those of Demazure  for  the Chow group \cite{Dem73} and for the Grothendieck group \cite{Dem74}. They provide algebraic tools to study oriented cohomology of homogeneous varieties. For instance, the $\gamma$-filtration  of $\hh(X)$ is defined using $c_F$, and the associated quotients $\gamma^{(d)}\hh(X)$ are studied in \cite{MZZ}. More precisely, it shows in \textit{loc.it.} that  $\gamma^{(d)}\hh(X)$ is torsion free, provided that the torsion index $\ttt$ of $G$ is  invertible in $R$. This does not include the case when $2$ is not invertible in $R$ and $G$ is of type $B_n$ and $D_n$. The goal of this paper is to study this case. More precisely, our main result is

\begin{theorem} \label{thm:main}
  Let $G$ be split, simple simply connected  of type  $B_n$ with $n\ge 3$ or of type $D_n$ with $n\ge 4$, and let  $X$ be its variety of complete flags. Let $\hh$ be a weakly birationally invariant oriented cohomology theory with coefficients in $R$ satisfying Assumption \ref{assump}. Suppose that $2$ is regular in $R$ but $\frac12\not\in R$.  Let $F$ be the  corresponding formal group law over $R$, and let $d\ge 2$.
  \begin{enumerate}
  \item[(i)] If $R$ has characteristic zero, then the torsion part of $\gamma^{(d)}\hh(X)$ is annihilated by
$\zeta_d^2\eta_d^2$, where the integers $\zeta_d$ and $\eta_d$ are defined in \ref{def:number}.
  \item[(ii)] Let   $d=2$ or $3$ and $F\equiv F_a\mod 2$. Then $\gamma^{(d)}\hh(X)\cong \gamma^{(d)}\CH(X;R)$. In particular, if $R$ has characteristic zero, then $\gamma^{(d)}\hh(X)$ is torsion free.
  \end{enumerate}

\end{theorem}
Note that the annihilator we obtained depends only on the filtration degree $d$. It does not depend  on the rank of $G$, nor on the specific cohomology theory $\hh$. The cohomology theories satisfying the hypothesis of Theorem \ref{thm:main} include any oriented cohomology theory over a field $k$ of characteristic zero (see \ref{remark:assumption}) such that  $2$ is regular but not invertible in $\hh(k)$, e.g., the algebraic cobordism. 

 For $d=1$, following the argument in \cite[Corollary 8.8]{MZZ}, it is easy to see that $\gamma^{(1)}\hh(X)\cong \gamma^{(1)}\CH(X)$, so it is always torsion free (when the characteristic of $R$ is zero). That's the reason why we restrict that $d\ge 2$. On the other hand, similar result was proved in   \cite[Theorem 6.1]{BZZ} for Chow group of twisted flag varieties of type $B_n$ and $D_n$. Note that for $\hh=K_0$ and $R=\Z$, the formal group law is $F(x,y)=x+y-xy$, so it does not satisfy the hypothesis of Theorem \ref{thm:main}.(ii). Therefore, our result does not contradict  \cite[Theorem 3.1]{GZ}, which  says that the torsion part of $\gamma^{(2)}K_0(X)$ is $\Z/2$. Corollary \ref{cor:K} provides more precise application to the Grothendieck group.

To prove Theorem \ref{thm:main}, we study $\ker c_F$ and the ideal $\IF^W$ of $R[[\La]]_F$ generated by nonconstant $W$-invariants.  The ideal $\IF^W$ itself has classical meanings. For example, $\calI_{F_a}^W$ is generated by the basic polynomials invariants \cite{Hum}, and  a theorem of Chevalley says that  $\calI_{F_m}^W$ is generated by the fundamental representations of $G$. On the other hand,  $\IF^W\subset \ker c_F$, and they coincide when the torsion index of $G$ is invertible in $R$.
We study the generators of $\IF^W$ and the  index of the embedding of $\IF^W$ in $\ker c_F$. We then use the deformation map \cite{MZZ} between formal group algebras of two distinct $F$ and $F'$ to define a map between $\gamma$-filtrations of corresponding oriented cohomologies $\hh$ and $\hh'$. Such map enables us to compare arbitrary $\hh$ with $\CH$.

This paper is organized as follows: In Section 2 we recall the definition of the formal group algebra $R[[\La]]_F$ and  the deformation map. In Section 3 we recall the definition of characteristic map and   $\gamma$-filtration. In Section 4 we study the generators of $\IF^W$. In Section 5 we provide an upper bound of the index of the embedding of $\IF^W$ in $\ker c_F$. In Section 6 we use the deformation map and the results in Sections 4 and 5 to define a map between  the $\gamma$-filtrations of different oriented cohomologies, and prove Theorem \ref{thm:main}.

\medskip

Through this paper, we adopt:

\noindent
\begin{itemize}
\item $R$ is a commutative ring with  identity such that $2$ is regular but not invertible.
\item  $G$ is a split simple simply connected linear algebraic group of classical Dynkin type $B_n$ with $n\ge 3$ or type $D_n$ with $n\ge 4$.
\item $\ttt$ is the torsion index of $G$, which is a power of $2$ in this case (\cite{Dem73}, \cite{Tot}).
\item $X$ is the variety of complete flags of $G$.
\item  $W$ is the Weyl group of $G$.
\item   $\La$ is  the group of characters of a maximal torus of $G$, which corresponds to the weight lattice of $G$.
\item $\{\om_1,...,\om_n\}$ is the set of fundamental weights, which is a basis of $\La$.
\item $\Sigma$ is the set of roots with a fixed set of simple roots $\Pi=\{\al_1,...,\al_n\}$.
\end{itemize}


\section{The formal group  algebra and the deformation map}
In this section we recall the definition of formal group algebra in \cite{CPZ} and  the deformation map in \cite{MZZ}.
Recall that a one-dimensional commutative formal group law $F$ over $R$ is a power series
$$F(x,y)=x+y+\sum_{i,j\ge 1}a_{ij}x^iy^j $$
with $a_{ij}\in R$ 
satisfying the following conditions:
$$F(x,F(y,z))=F(F(x,y),z), ~ F(x,y)=F(y,x),~ F(x,0)=x.$$
We use the notations $x+_Fy=F(x,y)$, $2\cdot_F x=F(x,x)$ and $ 3\cdot_F x=F(x,2\cdot x)$, etc.

\begin{example}\label{ex:fgl}\begin{itemize}
\item[(1)]  The \textit{addivitive formal group law} $F_a$ is defined by  $F_a(x,y)=x+y$.

\item[(2)]  The \textit{multiplicative formal group law} $F_m$ is defined by $F_m(x,y)=x+y-axy$ with $a\in R^\times$.

\item[(3)] The \textit{Lorentz formal group law} is defined by 
$$F_l(x,y)=\frac{x+y}{1+\beta xy}=(x+y)\sum_{i=0}(-\beta xy)^i, \quad \beta\neq 0\in R.$$

\item[(4)] \cite[\S IV.1]{Sil} Let $E$ be the elliptic curve defined by 
$$y=x^3+a_1xy+a_2x^2y+a_3y^2+a_4xy^2+a_6y^3,$$
then the \textit{elliptic formal group law} over $R=\Z[a_1,a_2,a_3,a_4,a_6]$ is defined by 
$$F_e(x,y)=x+y-a_1xy-a_2(x^2y+xy^2)-2a_3(x^3y+xy^3)+(a_1a_2-3a_3)x^2y^2+\ldots .$$
\end{itemize}
\end{example}

\begin{definition}  Let $F$ be a formal group law over $R$. Consider the polynomial ring $R[x_{\La}]$ in the  variables $x_\la$ with $\la\in \La$. Let $$\ep:R[x_{\La}]\to R, ~ x_{\la}\mapsto 0$$ be the augmentation map, and let $R[[x_{\La}]]$ be the ($\ker\ep$)-adic completion of $R[x_{\La}]$. Let $\calJ_F$ be the closure of the ideal of $R[[x_\La]]$ generated by $x_0$ and elements of the form $x_{\la_1+\la_2}-F(x_{\la_1}, x_{\la_2})$ for all $\la_1, \la_2\in \La$. Here $x_0\in R[x_{\La}]$ is the element determined by the zero element of $\La$. The \textit{formal group algebra } $R[[\La]]_F$ is defined to be the quotient
$$R[[\La]]_F=R[[x_{\La}]]/\calJ_F.$$
The augmentation map  induces a ring homomorphism $ \ep: R[[\La]]_F\to R$ with kernel  $\IF$. Then we have a filtration of  $R[[\La]]_F$:
$$R[[\La]]_F=\IF^0\supseteq \IF^1\supseteq \IF^2\supseteq \cdots$$
and the associated graded ring
$$Gr_R(\La, F)\overset{def}=\bigoplus_{i=0}^{\infty}\IF^i/\IF^{i+1}.$$
\end{definition}
\begin{example} \label{ex:subquotient} By \cite[Lemma 4.2]{CPZ}, $Gr_R(\La, F)$ is isomorphic to the symmetric algebra $S_R^*(\La)$. The isomorphism maps $\prod x_{\la_i}$ to $\prod \la_i$. Indeed, $R[[\La]]_F$ is non-canonically isomorphic to $R[[x_{\om_1},...,x_{\om_n}]]$. 
\end{example}

The action of the  Weyl group $W$ on $\La$ induces a $W$-action on $R[[\La]]_F$.
 Let  $\IF^W$ be the ideal of $R[[\La]]_F$ generated by the subset of constant-free $W$-invariants $R[[\La]]_F^W\cap \IF$. For $d\ge 0$, let 
 \begin{eqnarray*}\IF^{(d)}& = & \IF^d/\IF^{d+1},\\
 (R[[\La]]_F^W)^{(d)}&=& (R[[\La]]_F^W\cap \IF^d)/(R[[\La]]_F^W\cap \IF^{d+1}),\\
~(\IF^W)^{(d)}&=&(\IF^W\cap \IF^d)/(\IF^W\cap \IF^{d+1}).
 \end{eqnarray*}
 Then $\IF^{(d)}$ is a free $R$-module generated by $x_{\om_1}^{m_1}\cdot ...\cdot x_{\om_n}^{m_n}$ with $\sum_{i=1}^n m_i=d$.

\begin{ntt}
For any two formal group laws $F$ and $F'$ over $R$, there is an $R$-algebra isomorphism,  called the \textit{deformation map} from $F$ to $F'$
$$\Phi^{F\to F'}: R[[\La]]_F\to R[[\La]]_{F'}$$
defined as follows: firstly, one defines $\Phi^{F\to F'}(x_{\om_i})=x_{\om_i}\in R[[\La]]_{F'}$. For any ${\la}=\sum_{i=1}^nm_i\om_i\in\La$, we have $x_\la=x_{\sum m_i\om_i}\in R[[\La]]_F$. Then we define 
$$\Phi^{F\to F'}(x_\la)=(m_1\cdot_{F} x_{\om_1})+_{F}\cdots+_{F}(m_n\cdot _{F}x_{\om_n})\in R[[\La]]_{F'}.$$
Clearly $\Phi^{F'\to F}\circ \Phi^{F\to F'}=id_{R[[\La]]_F}$, so it is an isomorphism.
It maps $\IF^d$ into $\calI_{F'}^d$, hence induces an isomorphism of $R$-modules
$$\Phi_d^{F\to F'}: \IF^{(d)} \to \calI_{F'}^{(d)}.$$
A key property  is that for any $\Pi_{i=1}^dx_{\la_i}\in \IF^{(d)}$, we have 
\begin{equation}\label{eq:keyproperty}
\Phi_d^{F\to F'}(\Pi_{i=1}^dx_{\la_i})=\Pi_{i=1}^dx_{\la_i}\in \calI_{F'}^{(d)},
\end{equation} 
so $\Phi_d^{F\to F'}$ is $W$-equivariant. We then  have
$$\Phi_d^{F\to F'}:(\IF^{(d)})^W\overset{\cong}\longrightarrow
(\calI_{F'}^{(d)})^W,$$ but in general $ \Phi_d^{F\to F'}((\IF^W)^{(d)})$ is not contained in $ (\calI_{F'}^W)^{(d)}.$ One of the main interests of \cite[\S8]{MZZ} and Section 4 of this paper is the difference between  $\IF^W$ and $\mathcal{I}_{F'}^W$, i.e., to determine the smallest integer $\tau^{F\to F'}_d$ such that $$\tau^{F\to F'}_d\cdot (\calI_{F'}^W)^{(d)}\subset \Phi_d^{F\to F'}((\IF^W)^{(d)}).$$

If $R$ is a UFD, such integer exists and is called the $d$-th exponent of the $W$-action from $F$ to $F'$. In particular, by \cite{MZZ}, $\tau_d^{F_m\to F_a}$ coincides with the exponent $\tau_d$ defined in \cite{BNZ}, so $\tau_d^{F_m\to F_a}|2$ if $G$ is of type $B_n$ (resp. $D_n$) and $d\le 2n-1$ (resp. $d\le 2n-3$) by \cite{BZZ}. 
\end{ntt}

\section{The $\gamma$-filtration of oriented cohomology theory}

In this section we recall the definition of characteristic map and the $\gamma$-filtration of oriented cohomology theory of variety of complete flags \cite{MZZ}.

\begin{ntt} An algebraic oriented cohomology theory $\hh$( in the sense of Levine--Morel)  is a contravariant functor from the category of smooth projective varieties over a field $k$ to the category of commutative  (graded) $R$-algebras such that $\hh(\Spec k)=R$. It is  characterized by the axioms in \cite[\S1.1]{LM}. For instance, there exists  push-forward for projective morphism,   and  the projective bundle property and the extended homotopy property hold.

A cohomology theory is \textit{birationally invariant} \cite[Definition 8.7]{CPZ} if for any proper birational morphism $f:X\to Y$ of  smooth projective varieties, the push-forward of the fundamental class $f_*(1_{X})$ is $1_Y\in \hh(Y)$, and is \textit{weakly birationally invariant} if  $f_*(1_X)$  is invertible in $\hh(Y)$. The Chow ring $\CH$ over arbitrary base field is birationally invariant, and by \cite[Theorem 4.3.9]{LM}, the connective K-theory defined over a field of characteristic 0 is universal among all birationally invariant theories. Moreover, if the base field has characteristic 0, all oriented cohomology theories in the sense of Levine--Morel are weakly birationally invariant \cite[Corollary 8.10]{CPZ}.

 Each oriented cohomology theory determines characteristic classes, that is, a collection of maps
$$c_i^\hh:K_0(X)\to \hh(X),~ i\ge 1$$
characterized by  properties \cite[Definition 1.1.2]{LM}. In particular,
for any two line bundles $\calL_1$ and $\calL_2$ over $X$ one has
$$c_1^\hh(\calL_1\otimes \calL_2)=F(\calL_1,\calL_2)\in \hh(X) $$
for some formal group law $F$ over $R$.
This defines a map from the set of oriented cohomology theories to the set of one-dimensional commutative formal group laws.  For example, $F_a$ corresponds to the Chow group $\CH$ and $F_m$ corresponds to the Grothendieck group $K_0$.
\end{ntt}

\begin{ntt}
From now on, let $X$ be the variety of complete flags, and fix a Borel subgroup $B$ of $G$. If $G$ is of type $B_n$ ($n\ge 3$) or of type $D_n$ ($n\ge 4$), then the torsion index $\ttt$   is a power of $2$ (see \cite{Dem73} for definition and \cite{Tot} for computations).

Let $F$ be the formal group law corresponding to the oriented cohomology $\hh$, then there is a characteristic map, which is an $R$-algebra homomorphism
$$c_F:R[[\La]]_F\to \hh(X)$$
defined by $c_F(x_\la)=c_1^\hh(\calL(\la)).$ Here $\calL(\la)$ is the line bundle over $X$ corresponding to the character $\la$.

\begin{definition} \cite[p.9]{MZZ}
The $\gamma$-filtration of $\hh(X)$ is defined as follows:
  $\gamma^d\hh(X)$ is defined to be the  $R$-submodule of $\hh(X)$ generated by 
  $$ c_1^\hh(\calL_1)\cdot ...\cdot c_1^\hh(\calL_m)$$
  with $m\ge d$ and $ \calL_1,...,\calL_m$   line bundles over $ X.$ 
 Define $$\gamma^{(d)}\hh(X)=\gamma^d \hh(X)/\gamma^{d+1} \hh(X).$$
\end{definition}

By definition,  $c_F$ induces  maps
$$c_F:\IF^d\twoheadrightarrow \gamma^d\hh(X)~\text{ and }~ c_F^{(d)}: \IF^{(d)} \twoheadrightarrow \gamma^{(d)}\hh(X).$$

The Bruhat decomposition gives $X=\sqcup_{w\in W}BwB/B$, i.e., $X$ is a disjoint union of affine spaces. The closure of $BwB/B$ is denoted by $X_w$ and is called a Schubert variety. 
For any simple root $\al_i$, let $P_i$ be the minimal parabolic subgroup corresponding to $\al_i$. For any $w\in W$ and $I_w=(i_1,...,i_r)$ a reduced decomposition of $w$, the Bott--Samelson variety is defined as:
$$X_{I_w}:=P_{i_1}\times^B\cdots\times^B P_{i_r}.$$
The multiplication map induces $q_{I_w}:X_{I_w}/B\to X$ which factors through $X_w$:  $$q_{I_w}:X_{I_w}/B\to X_w\to X$$
 where the first map is surjective and birational, and the second one is a closed embedding. Denote $\zeta_{I_w}:=(q_{I_w})_*(1)\in \hh(X)$. 

\end{ntt}

\begin{assumption}\label{assump}\cite[Assumption 13.2]{CPZ} For each $w\in W$, let $I_w$ be a reduced decomposition of $w$. The set $\{\zeta_{I_w}\}_{w\in W}$ forms a $R$-basis of $\hh(X)$.
\end{assumption}

\begin{ntt}\label{remark:assumption}
For example, according to \cite[Lemma 13.3]{CPZ}, $\CH$ and $K_0$ in $\Z$ or $\Z/m$ coefficients over arbitrary base field satisfy this assumption, and so does any oriented cohomology theory over a field of characteristic zero. 
If, in addition,  $\hh$ is weakly birationally invariant and $\ttt$ is regular in $R$, then by \cite[Theorem 13.12]{CPZ}, $c_F$ coincides with the characteristic map defined in \cite[\S6]{CPZ} (one can view the latter map as the algebraic replacement of $c_F$). In this case,  $\IF^W\subset\ker c_F$. Furthermore, if the torsion index $\ttt$ is invertible in $R$, then  $c_F$ is surjective with $\ker c_F=\IF^W$ \cite[Theorem 6.9]{CPZ}.

Throughout this paper, we always assume that $\hh$ is weakly birationally invariant and satisfies Assumption \ref{assump}, for example, $\hh$ can be any oriented cohomology theory over $k$ with characteristic zero.
\end{ntt}

\section{The invariants}
In this section, we study the generators of $\IF^W$, and  prove Lemma \ref{lemma:key1} and \ref{lemma:invariant} concerning the invariants $\Th_d$. The ``only if'' parts of the two lemmas are proved in Lemmas 8.3, 8.4,  8.5 of \cite{MZZ}.

First, we prove some property of formal group law. Let
$$F(x,y)=x+y+\sum_{m=1}^\infty a_{mm}x^my^m+\sum_{l=3}\sum_{j+k=l, j< k}a_{jk}(x^jy^k+x^ky^j).$$
We use $\imath_F(x)\in R[[\La]]_F$ to denote the (formal) inverse of $x\in R[[\La]]_F$, and  $O(s)$ to denote a sum of terms of degree $\ge s$.
\begin{definition}\label{def:even} We say that a formal group law $F$ is \textit{even} if $$F(x,y)\equiv x+y\mod 2.$$

\end{definition}
\begin{example} 
\begin{itemize}
\item[(1)] If $\frac{\beta}{2}\in R$, then the Lorentz formal group law  $F_l(x,y)=\frac{x+y}{1+\beta xy}$ is even. 

\item[(2)]
If all the elements $a_1,a_2,a_3,a_4$ and $a_6$ in Example \ref{ex:fgl}.(4) are even integers, then the elliptic formal group law  $F_e$ is even. This follows from the fact that all the coefficients of $F_e(x,y)$ (except for those of $x$ and $y$) are combinations of $a_i$, $i=1,2,3,4,6$.
\end{itemize}
\end{example}

\begin{lemma} \label{lemma:inverse}
If the formal group law $F$ satisfies that $2|a_{mm}$ for $1\le m < s$, then
\[
  \imath_F(x)
\equiv
  x +  a_{ss} x^{2s} + O(2s+1)
\mod 2.
\]
Consequently, if $2|a_{ss}$  for all $s$, then
$\imath_F (x) \equiv x \mod 2$.

\end{lemma}

\begin{proof}
In general, we have
\[
  \imath_F(x)
=
  -x +  a_{11} x^2+ O(3),
\]
so the lemma holds for $s=1$.

We proceed  by induction on $s$. Assume it holds for $s= k-1$, i.e., if  $2|a_{mm}$ for $m<k-1$, then
\[
  \imath_F (x)
\equiv
  x + a_{k-1,k-1}x^{2k-2}+b_0 x^{2k-1} + b_1 x^{2k} + O(2k+1)\mod 2.
\]
Now assume $s=k$, i.e., assume in addition that  $2| a_{k-1, k-1}$.
By the induction assumption, \[
  \imath_F (x)
\equiv
  x +b_0 x^{2k-1} + b_1 x^{2k} + O(2k+1)\mod 2.
\]
It suffices to show that $b_0\equiv 0 $ and
$b_1\equiv a_{kk}\mod 2$.
Modulo 2 and $O(2k+1)$, we have
\begin{eqnarray*}
  0
&\equiv &
  F(x, \imath_F(x))
\\
&\equiv &
  x + (x+ b_0 x^{2k-1} + b_1 x^{2k}) + a_{kk} x^k (x+ b_0 x^{2k-1} +
  b_1x^{2k})^k
\\
&\,&
   + \sum_{l=3}^{2k} \sum_{\begin{array}{c}i<j,\\i+j=l\end{array}}
   a_{ij}
   \left( x^i ( x + b_0 x^{2k-1} + b_1 x^{2k})^j
    +x^j (x + b_0 x^{2k-1} + b_1 x^{2k})^i \right).
\end{eqnarray*}
Now, modulo $O(2k+1)$, we have
$x^k ( x + b_0 x^{2k-1} + b_1 x^{2k})^k \equiv x^{2k}$
and for each $i+j\ge 3$, we have
\begin{eqnarray*}
  x^i(  x + b_0 x^{2k-1} + b_1 x^{2k})^j
&\equiv &
  x^i \sum_{j_1+j_2+j_3=j} \binom{j}{j_1,j_2,j_3} (x)^{j_1}
  (b_0x^{2k-1})^{j_2} (b_1x^{2k})^{j_3}
\\
&\equiv &
x^{i+j}.
\end{eqnarray*}
Therefore, modulo 2 and $O(2k+1)$, we have
\[
  0
\equiv 
  F(x, \imath_F (x))
\equiv 
  b_0 x^{2k-1} + b_1 x^{2k} +  a_{kk} x^{2k}.
\]
Hence, $b_0\equiv 0 \mod 2$ and $b_1\equiv a_{kk} \mod 2$.
\end{proof}

\begin{ntt}\label{notation:weight}
We now define some elements of $\IF^W$ which are possible candidates of the generators of $\IF^W$.
Let  $\{e_i\}_{i=1}^n$ be the standard basis of  $\bbR^n$ that defines the root system of $G$. The element $e_i$ belongs to $\La$, hence can be written as a linear combination of $\om_i$'s. If $G$ is of type $B_n$ with $n\ge 3$, then  
$$e_1=\om_1,~ e_i=\om_i-\om_{i-1} \text{ for } 2\le i\le n-1, \text{ and } ~ e_n=2\om_n-\om_{n-1}.$$
If $G$ is of type $D_n$ with $n\ge 4$, then 
$$e_1=\om_1, ~ e_i=\om_i-\om_{i-1} \text{ for } 2\le i\le n-2,$$
$$ e_{n-1}=\om_n-\om_{n-1},~ \text{ and } e_n=\om_n+\om_{n-1}-\om_{n-2}.$$

For $d=1,...,n$, define the $W$-invariant element $\Th_d\in R[[\La]]_F^W\cap \IF$ together with a positive integer $r_d$ as follows;

\begin{enumerate}
    \item If $G$ is of type $B_n$ with $n\ge 3$, define $\Th_d^B=\sum_{i=1}^nx_{e_i}^dx^d_{-e_i}$. Since the Weyl group $W$ acts on $\{e_i\}_{i=1}^n$ by permutations and by sign changes, we see that $\Th^B_d\in R[[\La]]_F^W.$ Let $r_d=2$ if $d$ is a power of 2 and $r_d=1$ otherwise.
    \item If $G$ is  of type $D_n$ with $n\ge 4$, define $\Th^D_d=\Th^B_d$ for $d=1,...,n-1$ and $\Th^D_n=\prod_{i=1}^n(x_{e_i}-x_{-e_i})$. Since $W$ acts by permutations of ${e_i}$ and by sign changes of even numbers of $e_i$'s, we see that $\Th_d^D\in R[[\La]]_F^W$.  Let $r_n={2^n}$. For $d=1,...,n-1$, let $r_d=2$ if $d$ is a power of 2, and $r_d=1$ otherwise.
\end{enumerate}
\end{ntt}
\begin{example} By \cite[3.12]{Hum} and \cite[Remark 2 in page 19 and Ch. I. (2.4)]{Mac}, if $F=F_a$, then the coefficients of the polynomials $\Th_d\in R[[x_{\om_1},...,x_{\om_n}]]$ are integers with  g.c.d. $r_d$, and $$R[[\La]]_{F_a}^W=R[[\frac{1}{r_1}\Th_1,...,\frac{1}{r_n}\Th_n]].$$
But this may fail if $F\neq F_a$. In the following lemma, we will provide a necessary and sufficient condition for this to hold.
The idea of the proof is to express $x_{e_i}$ in terms of $x_{\om_j}$ using the relations in \ref{notation:weight}, and study their coefficients via the non-canonical isomorphism $R[[\La]]_F\cong R[[x_{\om_1},...,x_{\om_n}]]$.
\end{example}

\begin{lemma}\label{lemma:key1}
\begin{enumerate}
\item Let $G$ be of type $B_n$ with $n\ge 3$ (resp. of type $D_n$ with $n\ge 4$) and let $d\le n$ (resp. $d<n$) be a positive power of 2, then  $F$ is even if and only if $\frac{\Th_d}{2}\in \IF$ for some $d$ (hence for all $d$).
\item Let $G$ be of type $D_n$, then $2|a_{mm}$ for all $m\ge 1$ if and only if $\frac{\Th_n}{2^n}\in \IF$.
\end{enumerate}
\end{lemma}

\begin{proof}

(1) The ``only if'' part was proved in \cite{MZZ}, so we only prove the ``if'' part. Suppose $G$ is of type $B_n$.
 Let $\nu_0=0,$ $\nu_i=e_1+...+e_i=\om_i$ for $ i=1,...,n-1$ and
  $\nu_n=2\om_n$.
 We show that if $F$ is not even, then $2\nmid \Th^B_d$ for any $d$. Since $F$ is not even, then $2\nmid a_{ss}$ for some $s$ or $2\nmid a_{jk}$ for some $j< k$.

First, assume that $s$ is the smallest integer
  such that $2\nmid a_{ss}$.
  For any $\lambda_1 ,\lambda_2 \in \La$, let
  $x_{\lambda_1 - \lambda_2} = \sum_{k=1}^\infty f_k(x_{\lambda_1} ,
  x_{\lambda_2})$,
  where $f_k(x,y)$ is a homogeneous polynomial of degree $k$ in $R[x,y]$. For instance, $f_1(x,y)=x-y.$
  Since the binomial formula satisfies
  $(z_1+z_2)^{d}\equiv z_1^{d}+z_2^{d}\mod 2$, by Lemma
  \ref{lemma:inverse}, modulo $2$ and $\IF^{2sd+2d}$ we obtain
\begin{eqnarray*}
  \Th^B_d
&\equiv &
  \sum_{i=1}^nx^d_{e_i}x^d_{-e_i}
\equiv 
   \sum_{i=1}^nx^d_{e_i}(x_{e_i}+a_{ss}x^{2s}_{e_i})^d
\\
&\equiv &
  \sum_{i=1}^nx^d_{e_i}(x_{e_i}^d+a_{ss}^dx_{e_i}^{2sd})
\equiv 
   \sum_{i=1}^n(x_{e_i}^{2d}+a_{ss}^dx_{e_i}^{2sd+d})
\\
&\equiv  &
   \sum_{i=1}^n
  \Big[(\sum_{k=1}^\infty
  f_k(x_{\nu_i},x_{\nu_{i-1}}))^{2d}+a_{ss}^d
  (\sum_{k=1}^\infty f_k(x_{\nu_i}, x_{\nu_{i-1}}))^{2sd+d}\Big]
\\
&\equiv &
  \sum_{i=1}^n\Big[\sum_{k=1}^\infty
  f_k(x_{\nu_i},x_{\nu_{i-1}})^{2d}+a_{ss}^d
  (\sum_{k=1}^\infty f_k(x_{\nu_i}, x_{\nu_{i-1}}))^{2sd+d}\Big].
\end{eqnarray*}
Notice that $f_k(x_{\nu_i}, x_{\nu_{i-1}})^{2d}$ is a homogeneous
polynomial of degree $2kd$.
Therefore, the degree $(2s+1)d$ term of $\Th_d^B$ is given by
\[
~
  \sum_{i=1}^na_{ss}^df_1(x_{\nu_i},
  x_{\nu_{i-1}})^{2sd+d}=
  \sum_{i=1}^na_{ss}^d(x_{\nu_i}-x_{\nu_{i-1}})^{2sd+d}.
\]
Since $2ds+d$ is not a power of 2, by Lucas' Theorem, $2\nmid \begin{pmatrix}2ds+d \\ a\end{pmatrix}$ for some $0<a<2ds+d$, so $2\nmid (x_{\nu_i}-x_{\nu_{i-1}})^{2sd+d}$ for all $i$. Since $2\nmid a_{ss}$, so we have $2\nmid \Th^B_d$ in
$\IF/\IF^{2sd+d}$, which implies that
$2\nmid \Th^B_d$ in $\IF$.

Suppose that $2| a_{ss}$ for all $ s\ge 1$ and $l_0$ is the smallest integer such that $2\nmid a_{j_0,l_0-j_0}$ for some $j_0<l_0/2$. Then we can write
$$F(x,y)\equiv x+y+\sum_{l=l_0}\sum_{j< k, j+k=l} a_{jk}(x^jy^k+x^ky^j)\mod 2.$$
By Lemma \ref{lemma:inverse}, $x_{-e_i}\equiv x_{e_i}\mod 2$. So  $\Th^B_d\equiv \sum_{i=1}^{n+1}x_{e_i}^{2d} \mod 2$ in $\IF.$
Modulo 2, we have
\begin{eqnarray*}
     \Th^B_d
     &\equiv &
     \sum_{i=1}^n x_{e_i}^{2d}
     \equiv 
     \sum_{i=1}^n x_{\nu_i-\nu_{i-1}}^{2d}\\
     &\equiv &
     \sum_{i=1}^n F(x_{\nu_i}, \imath_F (x_{\nu_{i-1}}))^{2d}
     \equiv 
     \sum_{i=1}^n F(x_{\nu_i},x_{\nu_{i-1}})^{2d}\\
     &\equiv &
     \sum_{i=1}^n \left(x_{\nu_i}+x_{\nu_{i-1}}+\sum_{l=l_0}\sum_{j< k, j+k=l}a_{jk}(x_{\nu_i}^jx_{\nu_{i-1}}^k+x_{\nu_i}^kx_{\nu_{i-1}}^j)\right)^{2d}\\
     &\equiv &
     \sum_{i=1}^n\left(x^{2d}_{\nu_i}+x^{2d}_{\nu_{i-1}}+\sum_{l=l_0}
     \sum_{j< k, j+k=l}a_{jk}^{2d}(x^{2jd}_{\nu_i}x^{2kd}_{\nu_{i-1}}+x^{2kd}_{\nu_i}x^{2jd}_{\nu_{i-1}})\right).
\end{eqnarray*}
The coefficient of $x_{\nu_1}^{2j_0d}x_{\nu_2}^{2d(l_0-j_0)}=x_{\om_1}^{2j_0d}x_{\om_2}^{2d(l_0-j_0)}$ is $a_{j_0,l_0-j_0}^{2d}$, which is not divisible by 2 by assumption. So $2\nmid \Th^B_d$.

If the root system is of type $D_n$ and $d<n$, the proof is similar.

(2) If $2|a_{mm}$ for all $m$, then $x_{-e_i}\equiv x_{e_i}\mod 2$ by Lemma \ref{lemma:inverse}. So $2|(x_{e_i}-x_{-e_i})$ and $2^n|\Th^D_n$.

Conversely, if $2\nmid a_{mm}$ for some $m$, let $s\ge 1$ be  the smallest such integer. Then by Lemma \ref{lemma:inverse}, 
$x_{-e_i}\equiv x_{e_i}+a_{ss}x_{e_i}^{2s}\mod 2,$ so $x_{e_i}-x_{-e_i}+a_{ss}x_{e_i}^{2s}\equiv 0\mod 2.$ Let $\varrho_i=x_{e_i}-x_{-e_i}$, then modulo $2^n$, 
\begin{eqnarray*}
  0
  &\equiv &
  \prod_{i=1}^n (\varrho_i+a_{ss}x_{e_i}^{2s})\\
  &\equiv &
  \prod_{i=1}^n\varrho_i+a_{ss}\sum_{i=1}^nx_{e_i}^{2s}S_{n-1}(\{\varrho_j\}_{j\neq i})+O(n+2s).\\
\end{eqnarray*}
Here $S_{n-1}$ is the elementary symmetric polynomial of degree $n-1$. We  compute $x_{e_1}^{2s}S_{n-1}(\{\varrho_j\}_{j\neq 1})$, which has degree $n+2s-1$. Modulo $\IF^{n+2s}$, it is reduced to the additive case, in which case $\varrho_j=2x_{e_j}$. Hence
$$x_{e_1}^{2s}S_{n-1}(\{\varrho_j\}_{j\neq 1})\equiv x_{e_1}^{2s}S_{n-1}(\{2x_{e_j}\}_{j\neq 1})\equiv 2^{n-1}\cdot x_{e_1}^{2s} S_{n-1}{(\{x_{e_j}\}_{j\neq 1})} \mod \IF^{n+2s}.$$
Representing $\{x_{e_i}\}_{i=1}^n$ by $\{x_{\om_j}\}_{j=1}^n$ using the relations in \ref{notation:weight}, we see that the coefficient of the monomial  $x_{\om_1}^{2s+1}\prod_{j=2}^{n-1}x_{\om_j}$ is $(-1)^{n-1}$, so the g.c.d. of the coefficients of $x_{e_1}^{2s}S_{n-1}(\{\varrho_j\}_{j\neq 1})$ in $R[x_{\om_1},...,x_{\om_n}]$ is $2^{n-1}$. Hence, 
$$a_{ss}\sum_{i=1}^nx_{e_i}^{2s}S_{n-1}(\{\varrho_j\}_{j\neq i})\not\equiv 0\mod 2^n.$$ 
 So $\Th_n^D=\prod_{i=1}^n\varrho_i$ is not divisible by $2^n$ in $\IF/\IF^{n+2s}$, so $\frac{1}{2^n}\Th_n^D\notin \IF.$   
  \end{proof}

\begin{ntt} 

If $f\in \IF^d\backslash \IF^{d+1}$, we say that $\deg f=d$. Then for $d=1,...,n$, we have $\deg \Th_d^B=2d$. For the type $D_n$, $\deg \Th_d^D=2d$ for $d=1,...,n-1$ and $\deg \Th_n^D=n$.
Given a $n$-tuple $\al=(\al_1,...,\al_n)$ with $\al_i\in \Z_{\ge 0}$, let $r_\al=\prod_{i=1}^nr_i^{\al_i}$, $\Th(\al)=\prod_{i=1}^n\Th_i^{\al_i}$ and $|\al|=\sum_{i=1}^n\al_i\cdot \deg \Th_i.$

\label{def:number}
Let $\nu_2(m)$ be the $2$-adic valuation of $m$. To simplify the notations, we define a collection of integers $\{\zeta_d, \eta_d\}_{d\ge 1}$ which depends on the Dynkin type of $G$.
\begin{enumerate}
     \item If $G$ is of type $B_n$ with $n\ge 3$, let $\zeta_d=2^{[d/2]}$ for $d\ge 1$. Let $\eta_1=1,$ $\eta_2=\eta_3=2,$ $\eta_4=4,$ $\eta_d=2^{d+\nu_2([d_0/2]!)}$ for $d\ge 5$ with $d_0\overset{def}=\min\{d,2n\}$.
     \item If $G$ is of type $D_n (n\ge 4)$, let $\zeta_d=2^{[d/2]}$ for $1\le d<n$ and $\zeta_d=2^{[d/n]n}$ for $d\ge n$. If  $n=4$, let $\eta_1=1,$ $\eta_2=\eta_3=2,$ $\eta_d=2^{d+\nu_2([d_1/2]!)}$ for $d\ge 4$ with $d_1\overset{def}=\min\{d,2n-2\}$. If $n\ge 5$, let $\eta_4=4$ and for other $d$,  define $\eta_d$ the same way as for $D_4$.
\end{enumerate}
The $\zeta_d$'s were defined in \cite{MZZ} and the $\eta_d$'s were defined in \cite{BZZ}.

\end{ntt}

\begin{lemma}\label{lemma:invariant}Let $G$ be of type $B_n$ with $n\ge 3$ or of type $D_n$ with $n\ge 4$ and let $d\ge 2$.
   \begin{enumerate}
   \item We have
   $$\zeta_d\cdot (R[[\La]]_F^W)^{(d)}\subseteq \langle\Th(\al)\rangle_{|\al|=d}\subseteq (R[[\La]]_F^W)^{(d)}.$$
   Moreover, $F$ is even if and only if for some $d$ (hence for all $d$), 
   $$(R[[\La]]_F^W)^{(d)}=\langle\frac{1}{r_\al}\Th(\al)\rangle_{|\al|=d}.$$
   \item We have
   $$\zeta_d\cdot (\IF^W)^{(d)}\subseteq \{\sum_{\deg \Th_i\le d} g_i\Th_i|g_i\in \IF^{(d-\deg \Th_i)}\}\subseteq (\IF^W)^{(d)}.$$
   Moreover, $F$ is even if and only if for some $d$ (hence for all $d$),
   $$(\IF^W)^{(d)}=\{\sum_{\deg \Th_i\le d}\frac{g_i}{r_i}\Th_i|g_i\in \IF^{(d-\deg \Th_i)}\}.$$
   \end{enumerate}
\end{lemma}
\begin{proof}
  (1) The first statement and the ``only if'' part of the second statement were proved in \cite[Lemma 8.4]{MZZ}. For the ``if" part of the second statement, note that the assumption $$(R[[\La]]_F^W)^{(d)}=\langle\frac{1}{r_\al}\Th(\al)\rangle_{|\al|=d}$$
 for some $d$ implies that $\frac{1}{2}\Th_1\in \IF$. By Lemma \ref{lemma:key1}, $F$ is even.

  (2) The first statement and the ``only if'' part of the second statement were proved in \cite[Lemma 8.5]{MZZ}. The proof of the ``if'' part is similar to that of (1).
\end{proof}
 
\begin{remark}\cite[Lemmas 8.4, 8.5, Theorem 8.6]{MZZ}\label{rem:invariant}
  Indeed, in Lemma \ref{lemma:invariant}, if one replaces the condition that $F$ is even by the condition that $\frac{1}{2}\in R$, then we have
  $$(R[[\La]]_F^W)^{(d)}=\langle\Th(\al)\rangle_{|\al|=d}, ~ (\IF^W)^{(d)}=\{\sum_{\deg \Th_i\le d}{g_i}\Th_i|g_i\in \IF^{(d-\deg \Th_i)}\}, $$
  and $(\calI^W)^{(d)}=\Phi_d^{F\to F'}((\IF^W)^{(d)})$ for arbitrary $F, F'$ and $d\ge 2$. Similarly,  if both $F$ and $F'$ are even, then one still has $(\calI_{F'}^W)^{(d)}=\Phi^{F\to F'}_d((\IF^W)^{(d)}).$ For general $F$ and $F'$, one has $\zeta_d\cdot (\calI_{F'}^W)^{(d)}\subset \Phi_d^{F\to F'}((\IF^W)^{(d)})$. We will use these facts in the next section.
\end{remark}

\section{The kernel of the characteristic map}\label{sec:CompInv}
In this section  we compute an upper bound of the index of embedding $(\IF^W)^{(d)}$ in $\ker c_F^{(d)}$, which will be used in Section 6 to prove the main result.

 Let $\tilde{R}=R[\frac{1}{2}]$. Let $\tilde{\calI}_F\subset \tilde R[[\La]]_F$ (resp. $\tilde c_F$) be the corresponding augmentation ideal (resp. the characteristic map). Let $c_F^{(d)}$ and $\tilde{c}_F^{(d)}$ be the  characteristic maps on the corresponding subquotients on $\IF^{(d)}$ and $\tilde \IF^{(d)}$, respectively. By  \ref{remark:assumption}, $ (\tilde\IF^W)^{(d)}= \ker \tilde c_F^{(d)}$.  
By \cite[Proposition 6.5]{CPZ},
there is a commutative diagram
\begin{equation}\xymatrix{&\ker c_F^{(d)} \ar@{^{(}->}[r] \ar[d] & \IF^{(d)} \ar[d]\ar@{->>}[r]^-{c_F^{(d)}}  & \gamma^{(d)}\hh(X)\ar[d]\\
   (\tilde\IF^W)^{(d)}\ar@{=}[r]    &       \ker \tilde{c}_F^{(d)}\ar@{^{(}->}[r] & \tilde{\IF}^{(d)}\ar@{->>}[r]^-{\tilde{c}_F^{(d)}}  & \gamma^{(d)}\tilde{\hh}(X).}
   \end{equation}
 For any $y\in \ker c_F^{(d)}$ we have $y\in \ker \tilde c_F^{(d)}=(\tilde \IF^W)^{(d)}$, so by Remark \ref{rem:invariant},
\begin{equation}\label{eq:kerinv}
y=\sum_{\deg \Th_i\le d}{g_i}\Th_i,~ g_i\in \tilde\IF^{(d-\deg\Th_i)}.
\end{equation}

The following two lemmas generalize \cite[\S1B]{GZ},  \cite[Lemma 6.4]{BNZ} and \cite[Proposition 4.5]{BZZ} from $F_a$ to general $F$. One also notes that if $F=F_m$, then $\ker c_{F}=\IF^W$.

\begin{lemma}\label{lemma:kerinvspec}
\begin{enumerate}
\item  Let $G$ be of type $B_n$ with $n\ge 3$ or of type $D_n$ with $n\ge 4$ and let $d=2$ or 3. Then $2\cdot \ker c_F^{(d)}\subseteq (\IF^W)^{(d)}$. If
   $F$ is even, then $\ker c_F^{(d)}=(\IF^W)^{(d)}$.
 \item Let $G$ be of type $B_n$ with $n\ge 3$ or of type $D_n$ with $n\ge 5$. Let $d=4$. We have $4\cdot \ker c_F^{(d)}\subseteq (\IF^W)^{(d)}$. If $F$ is even, then $2\cdot \ker c_F^{(d)}\subseteq (\IF^W)^{(d)}$.
 \end{enumerate}
\end{lemma}

\begin{proof}

(1) Suppose that $G$ is of type $B_n$. 
For any $y\in \ker c_F^{(2)}$,   by Equation (\ref{eq:kerinv}), we have
 $$y=u\cdot {\Th_1}\in (\tilde R[[\La]]_F^W)^{(2)}$$
 for some $u\in \tilde R.$
 That is, $u\cdot \Th_1=y $ in $\IF^{(2)}$, so both sides are polynomials of degree 2 in $R[x_{\om_1},...,x_{\om_n}]$. Note that
   $$\Th_1=2\sum_{i=1}^{n-2}(x_{\om_i}^2-x_{\om_i}x_{\om_{i+1}})+2x_{\om_{n-1}}^2-4x_{\om_{n-1}}x_{\om_n}+4x_{\om_{n}}^2.$$
   The g.c.d. of the coefficients of $\Th_1\in (\IF^W)^{(2)}$ is 2, so ${2u}\in R$. Therefore, $2y=2u\cdot \Th_1\in (\IF^W)^{(2)}.$

If $F$ is even, then by Lemma \ref{lemma:key1}, $\frac{\Th_1}{2}\in R[[\La]]_F^W$, so $y=\frac{\Th_1}{2}\cdot 2u\in (\IF^W)^{(2)}$. 

Now let $d=3$. 
For any $y\in \ker c_F^{(3)}$, by Equation (\ref{eq:kerinv}), we have 
\begin{equation}\label{eq:deg3}
y= \Th_1\cdot f_1 \in  \tilde \IF^{(3)}
\end{equation}
for some $f_1\in \tilde \IF^{(1)}.$
 We show that $2f_1\in \IF^{(1)}$. Suppose that  $f_1=\sum_{i=1}^n {a_i}x_{\om_i}$ with $ a_i\in \tilde R$. Write $$y= \sum_{i=1}^na_{iii}x_{\om_i}^3+\sum_{i<j}(a_{iij}x_{\om_i}^2x_{\om_j}+a_{ijj}x_{\om_i}x_{\om_j}^2)+
 \sum_{i<j<k}a_{ijk}x_{\om_i}x_{\om_j}x_{\om_k}\in  \IF^{(3)}$$
 with $a_{iii}, a_{ijj}, a_{iij}, a_{ijk}\in R$.
For any $i<n$, by comparing the coefficients of $x_{\om_i}^3$ in Equation (\ref{eq:deg3}), we see that $2a_i=a_{iii}\in R$. By comparing the coefficients of $x_{\om_1}^2x_{\om_n}$, we have $2a_n=a_{11n}\in R$. Hence, $2f_1\in \IF^{(1)}$ and $2y=2f_1\Th_1\in (\IF^W)^{(3)}$.

  If $F$ is even, then by Lemma \ref{lemma:key1}, $\frac{\Th_1}{2}\in R[[\La]]_F^W$, so $y=\frac{\Th_1}{2}\cdot 2f_1\in (\IF^W)^{(3)}$.

If $G$ is of type $D_n$ with $n\ge 4$, the proof is similar, since the generator  involved in this case is  $\Th_1$ only.

(2) Let $G$ be of type $B_n$. For $y\in \ker c_F^{(4)}$,   by Equation (\ref{eq:kerinv}),
  $$y=f_0\Th_2+f_2\Th_1$$
  for some polynomials $ f_i\in \tilde \IF^{(i)}.$
Notice that there exists  positive integer $b$ such that the polynomials $2^bf_0\in \IF$ and $2^bf_2\in \IF$. Let $b_0$ be the smallest among these integers, and we claim that $b_0\le 2$. If not, then  ${b_0}\ge 3$. It implies that
  $$2^{b_0}f_0\Th_2+2^{b_0}f_2\Th_1=2^{b_0}y\equiv 0\mod 8$$
  with $2^{b_0}f_i\in \IF^{(i)}.$
  Since $\frac{\Th_1}{2}\in \IF^{(2)}$ and $\frac{\Th_2}{2}\in \IF^{(4)}$, so  in $\IF^{(4)}$, we have $$ 2^{b_0}f_0\frac{\Th_2}{2}+2^{b_0}f_2\frac{\Th_1}{2}\equiv 0\mod 4.$$
 By Example \ref{ex:subquotient}, $\IF^{(4)}\cong \calI_a^{(4)}$. By the proof of \cite[Lemma 6.4]{BNZ}, this implies that $g.c.d.\{2^{b_0}f_0,2^{b_0}f_2\}=2$, therefore, $2^{{b_0}-1}f_i\in \IF^{(i)}$. This contradicts to the minimality assumption of $b_0$. Hence ${b_0}\le 2$ and  $4y=4f_0\Th_2+4f_2\Th_1\in (\IF^W)^{(4)}$.

  If $F$ is even, then $\frac{\Th_1}{2}, \frac{\Th_2}{2}\in  \IF$, therefore, $2y=4f_0\cdot \frac{\Th_2}{2}+4f_2\frac{\Th_1}{2}\in (\IF^W)^{(4)}$.

  If $G$ is of type $D_n$ with $n\ge 5$, the proof is similar, since the only generators of $(R[[\La]]_F^W)^{(4)}$ are $\Th_1$ and $\Th_2$.
\end{proof}

\begin{lemma}\label{lemma:kerinvgen}
  Let $G$ be of type $B_n$ with $n\ge 3$ or of type  $D_n$ with $n\ge 4$, then $\eta_d\cdot \ker c_F^{(d)}\subseteq (\IF^W)^{(d)}$, where the integer $\eta_d$ was defined in \ref{def:number}.
\end{lemma}
\begin{proof}Let $G$ be of type $B_n$. The case of type $D_n$ is similar. For  $d\le 4$, it is proved in Lemma \ref{lemma:kerinvspec}. So let $d\ge 5$.
For any $y\in \ker c_F^{(d)}$, by Equation (\ref{eq:kerinv}), 
  \begin{equation}\label{eq:ker}
  y=\sum_{\deg \Th_i\le d} f_{d-2i}\Th_i\in  \IF^{(d)},~ f_{d-2i}\in \tilde \IF^{(d-2i)}.
  \end{equation}
   The polynomials $f_{d-2i}\in \tilde\IF^{(d-2i)}$ are non-uniquely determined by $y$, and there exists positive integer $b$ (determined by $\{f_{d-2i}\}$) such that $2^bf_{d-2i}\in \IF^{(d-2i)}$ for all $i$. Suppose that $b_0$ is the smallest among these integers and among $\{f_{d-2i}\}$ satisfying Equation (\ref{eq:ker}). We claim that $2^{b_0}|\eta_d$. If not, then $2\eta_d|2^{b_0}$. Then in $\IF^{(d)}$, we have
$$2^{b_0}y= \sum_{\deg \Th_i\le d}2^{b_0}f_{d-2i}\Th_i\in \IF^{(d)},~ 2^{b_0}f_{d-2i}\in \IF^{(d-2i)}.$$
By Example \ref{ex:subquotient}, $\IF^{(d)}\cong \calI_a^{(d)}$.
By the proof of \cite[Proposition 4.5]{BZZ}, we know that there exists $g_{d-2i}\in \tilde\IF^{(d-2i)}$ such that
$$y= \sum_{\deg \Th_i\le d}g_{d-2i}\Th_i\in \IF^{(d)}$$
with $2^{{b_0}-1}g_{d-2i}\in \IF^{(d-2i)}.$ This contradicts  the minimality assumption of ${b_0}$.  Therefore, $2^{b_0}|\eta_d$ and $\eta_d y=\sum \eta_d f_{d-2i}\Th_i\in (\IF^W)^{(d)}$.
\end{proof}
\section{Comparison of $\gamma$-filtrations}
In this section we apply the computation in Sections 4 and 5 to compare $\gamma$-filtrations of different oriented cohomology theories, and prove the main result of this paper. 

\begin{lemma}\label{lemma:main}
Let $G$ be of type $B_n$ with $n\ge 3$ or of type $D_n$ with $n\ge 4$. Let $\hh$ and $\hh'$ be two weakly birationally invariant oriented cohomology theories  satisfying Assumption \ref{assump}. Let $F$ and $F'$ be the  corresponding formal group laws, respectively. Then:
\begin{itemize}
\item[(a)] The map $\zeta_d\eta_d\cdot \Phi_d^{F\to F'}:\IF^{(d)}\to \calI_{F'}^{(d)}$ induces a map $\gamma^{(d)}\hh(X)\to \gamma^{(d)}\hh'(X)$.
\item[(b)] 
 If $F$ and $F'$ are even, then one can replace $\zeta_d\eta_d$ in \textrm{(a)} by  $\eta_d$.
 \end{itemize}

\end{lemma}
\begin{proof}
(a) Suppose $G$ is of type $B_n$. The case of type $D_n$ can be proved similarly. 
  We have the following diagram
  $$\xymatrix{(\IF^W)^{(d)} \ar@{^{(}->}[r] &\ker c_F^{(d)}\ar@{^{(}->}[r] & \IF^{(d)} \ar[d]^{\Phi_d^{F\to F'}}_{\cong} \ar@{->>}[r]^-{c_F^{(d)}} & \gamma^{(d)}\hh(X)\\
  (\calI_{F'}^W)^{(d)}\ar@{^{(}->}[r] & \ker c_{F'}^{(d)}\ar@{^{(}->}[r] & \calI_{F'}^{(d)} \ar@{->>}[r]^-{c_{F'}^{(d)}} &\gamma^{(d)}\hh'(X).}$$
  It suffices to show that $\zeta_d\eta_d\cdot \Phi_d^{F\to F'}$ maps $\ker c_F^{(d)}$ into $\ker c_{F'}^{(d)}$.
For any $y\in \ker c_F^{(d)}$, by Lemma \ref{lemma:kerinvgen}, $\eta_d\cdot y\in (\calI_{F}^{W})^{(d)}$. By Lemma \ref{lemma:invariant}, $$\zeta_d\eta_d\cdot y=\sum_{\deg \Th_i\le d} g_i\Th_i$$
for some $ g_i\in \IF^{(d-2i)}.$ 
By Equation (\ref{eq:keyproperty}), we have 
$$\Phi_d^{F\to F'}(\zeta_d\eta_d\cdot y)=\sum_{\deg \Th_i\le d}g_i\Th_i\in (\calI_{F'}^W)^{(d)}\subseteq \ker c_{F'}^{(d)}.$$
Therefore, $\zeta_d\eta_d\cdot \Phi_d^{F\to F'}$ induces a map $\gamma^{(d)}\hh(X)\to \gamma^{(d)}\hh'(X)$.

(b) 
If $F$ and $F'$ are even, then for any $y\in \ker c_{F}^{(d)}$,  by Lemma \ref{lemma:kerinvgen}, $\eta_d\cdot y\in (\IF^W)^{(d)}.$
By Remark \ref{rem:invariant},  $\Phi_{d}^{F\to F'}((\IF^W)^{(d)})=(\calI_{F'}^W)^{(d)}$. Hence,  $$\Phi_d^{F\to F'}(\eta_d\cdot y)\in (\calI_{F'}^W)^{(d)}\subseteq \ker c_{F'}^{(d)}.$$
Therefore,  $\eta_d\cdot \Phi_d^{F\to F'}$ induces a map $\gamma^{(d)}\hh(X)\to \gamma^{(d)}\hh'(X)$.
\end{proof}

We are now ready to prove the main result of this paper.

\begin{proof}[Proof of Theorem \ref{thm:main}] We only consider the $B_n$ case, since the $D_n$ case is similar.

 (i) By Lemma \ref{lemma:main}, there is a commutative diagram
   \begin{equation}\label{diagram:char}
   \xymatrix{\IF^{(d)} \ar[d]_-{\zeta_d\eta_d\Phi_d^{F\to F_a}}\ar@{->>}[r]^-{c_{F}^{(d)}} & \gamma^{(d)}\hh(X) \ar[d]\\
           \calI_{a}^{(d)}\ar@{->>}[r]^-{c_{a}^{(d)}} & \gamma^{(d)}\CH(X;R).}
           \end{equation}
   Given any torsion element $u\in \gamma^{(d)}\hh(X)$, since $\gamma^{(d)}\CH(X;R)\subseteq \CH^d(X;R)$ is torsion free, so $u$ is mapped to 0 in $\gamma^{(d)}\CH(X;R)$. Lift $u$ to an element $v\in \IF^{(d)}$, and look at its image $\zeta_d\eta_d\Phi_d^{F\to F_a}(v)\in \calI_a^{(d)}$. Since $c_a^{(d)}(v)=0$, so $\zeta_d\eta_d\Phi_d^{F\to F_a}(v)\in \ker c_a^{(d)}$, hence by Lemma \ref{lemma:kerinvgen}, $\eta_d\zeta_d\eta_d\Phi_d^{F\to F_a}(v)\in (\calI_a^W)^{(d)}$, and by Remark \ref{rem:invariant}, $$\zeta_d\eta_d\zeta_d\eta_d\Phi_d^{F\to F_a}(v)\in \Phi_d^{F\to F_a}((\IF^W)^{(d)}).$$ Applying $(\Phi_d^{F\to F_a})^{-1}$, we see that $\zeta_d^2\eta_d^2v\in (\IF^W)^{(d)}\subseteq \ker c_F^{(d)}$. Hence, $\zeta_d^2\eta_d^2\cdot u=c_F^{(d)}(\zeta_d^2\eta_d^2\cdot v)=0$.

(ii)  Let $d=2$ or 3. Since $F$ and $F_a$ are even, so by Lemma \ref{lemma:kerinvspec},  $$(\IF^W)^{(d)}=\ker c_F^{(d)}~\text{ and }~(\calI_{a}^W)^{(d)}=\ker c_{a}^{(d)}.$$
By Remark \ref{rem:invariant}, we know that  $\Phi_d^{F\to F_a}((\IF^W)^{(d)})=(\calI_a^W)^{(d)}$. Therefore,   the isomorphism $\Phi_d^{F\to F_a}$ restricted to $\ker c_F^{(d)}$ induces an isomorphism
  $$\ker c_F^{(d)}\cong \ker c_{a}^{(d)},$$
  hence it induces an isomorphism
  $$\gamma^{(d)}\hh(X)\cong \gamma^{(d)}\CH(X;R).$$
\end{proof}

\begin{remark}In Theorem \ref{thm:main}.(i), if $F$ is even, then one can use Lemma \ref{lemma:main}.(b) to replace $\zeta_d^2\eta_d^2$ by $\zeta_d\eta_d$. 
\end{remark}
\begin{corollary}\label{cor:K} If $F$ is the corresponding formal group law for $\hh$, then the map $\zeta_d\cdot \Phi^{F_m\to F}_d$ induces a map  $\gamma^{(d)}K_0(X)\to \gamma^{(d)}\hh(X)$. In particular, if $R=\Z$, then the torsion part of $\gamma^{(d)}K_0(X)$ is annihilated by $\zeta^2_d\eta_d$.
\end{corollary}
\begin{proof} The proof is similar to those of Lemma \ref{lemma:main}.(a) and Theorem \ref{thm:main}.(i) by using the fact that $\ker c_{F_m}=\calI_{F_m}^W$.
\end{proof}
\begin{remark}This corollary can be used to refine the upper bound in \cite{BZZ} of the annihilator of Chow group of twisted flag varieties.
\end{remark}

\paragraph{\bf Acknowledgments}

The author is supported by the Fields Institute and NSERC grant of Kirill Zainoulline. The author would like to thank Kirill Zainoulline for suggesting  this topic, and thank Jos\'{e} Malag\'{o}n-L\'{o}pez for helpful discussion. He would also like to thank the referee for helpful suggestions.

\end{document}